\documentclass[a4paper,11pt]{article}
\usepackage{graphicx}
\usepackage{amssymb}
\usepackage[french, english]{babel}
\usepackage[utf8]{inputenc}
\usepackage{amsmath,amsfonts,amssymb,amsthm}
\usepackage[T1]{fontenc}
\usepackage{fullpage}
\usepackage{hyperref}
\usepackage{mathrsfs}
\usepackage{relsize}
\usepackage{tikz}
\usepackage{bbm}
\usepackage{appendix}
\usepackage{pifont}
\usetikzlibrary{patterns}

\definecolor{darkgreen}{rgb}{0.0, 0.5, 0.0}

\newcommand{\loc}{\mathrm{loc}}

\newcommand{\R}{\mathbb R}

\newcommand{\E}{\mathbb E}

\newcommand{\B}{\mathcal{B}}

\renewcommand{\P}{\mathbb P}

\newcommand{\occ}{\mathrm{occ}}
\newcommand{\ff}{\mathbf{f}}
\newcommand{\q}{\mathbf{q}}
\newcommand{\tq}{\mathbf{\widetilde{q}}}

\newcommand{\QQ}{\mathbf{Q}}
\newcommand{\oQQ}{\mathbf{Q'}}

\newcommand{\hQ}{\mathcal{Q}_h}

\newcommand{\MM}{\mathbb{M}}
\newcommand{\oMM}{\mathbb{M}'}

\renewcommand{\v}{\mathbf{v}}
\newcommand{\ov}{\mathbf{v'}}

\newcommand{\op}{p'}

\newcommand{\om}{m'}
\newcommand{\oM}{M'}
\renewcommand{\oe}{e'}

\newcommand{\tf}{\widehat{f}}
\newcommand{\tg}{\widehat{g}}

\newcommand{\diam}{\mathrm{diam}}

\theoremstyle{definition}

\newtheorem{thm}{Theorem}

\newtheorem{defn}{Definition}%[section]

\newtheorem{prop}[defn]{Proposition}

\newtheorem{lem}[defn]{Lemma}

\tikzstyle{every node}=[circle, draw, fill=black!50, inner sep=0pt, minimum width=4pt]
\tikzstyle{white}=[circle, draw, fill=black!0, inner sep=0pt, minimum width=4pt]
\tikzstyle{bigwhite}=[circle, draw, fill=black!0, inner sep=0pt, minimum width=10pt]
\tikzstyle{dual}=[circle, draw=blue, fill=black!0, inner sep=0pt, minimum width=4pt]
\tikzstyle{fat}=[circle, draw, fill=red!50, inner sep=0pt, minimum width=8pt]
\tikzstyle{fat_bis}=[circle, draw, fill=blue!50, inner sep=0pt, minimum width=8pt]
\tikzstyle{fat_ter}=[circle, draw, fill=green!50, inner sep=0pt, minimum width=8pt]
\tikzstyle{rouge}=[circle, draw, fill=red, inner sep=0pt, minimum width=7pt]
\tikzstyle{bleu}=[circle, draw, fill=blue, inner sep=0pt, minimum width=7pt]
\tikzstyle{petitrouge}=[circle, draw, fill=red, inner sep=0pt, minimum width=4pt]
\tikzstyle{petitbleu}=[circle, draw, fill=blue, inner sep=0pt, minimum width=4pt]
\tikzstyle{texte}=[draw=none, fill=none]

\title{\bf{Local limits of random bipartite maps in high genus: the general case}}
\author{Thomas \bsc{Budzinski}\footnote{CNRS and ENS de Lyon, \url{thomas.budzinski@ens-lyon.fr}}}

\begin{document}

\maketitle

\begin{abstract}
	We prove the local convergence of uniform bipartite maps with prescribed face degrees in the high genus regime. Unlike in the previous work~\cite{BL22} on the subject, we do not make any assumption on the tail of the face degrees, except that they remain finite in the limit.
\end{abstract}

\section{Introduction}

%\paragraph{Local limits of random maps in high genus.}
Infinite local limits of random planar maps have been the object of extensive literature since the introduction of the UIPT by Angel and Schramm~\cite{AS03}. In particular, local limits of Boltzmann bipartite planar maps are now very well understood~\cite{Bud15, BC16} (we also refer to~\cite{C-StFlour} for a complete survey and to~\cite{St18} for an investigation of the non-bipartite case). On the other hand, the study of random maps whose genus goes to infinity at the same time as their size is much more recent. It was proved in~\cite{BL19} that uniform high genus triangulations converge locally to a model of random infinite hyperbolic triangulations introduced earlier by Curien~\cite{CurPSHIT}. This result was extended to a large class of bipartite maps in~\cite{BL22} (with the limiting objects defined in~\cite[Appendix C]{B18these}) under the assumption that the expected degree of a typical face remains finite in the limit. The goal of this note is to get rid of this assumption.

More precisely, let $\mathbf{f}=(f_j)_{j \geq 1}$ be a face degree sequence (i.e. a sequence of nonnegative integers such that $f_j=0$ for $j$ large enough), and let $0 \leq g \leq \frac{1}{2} \sum_{j \geq 1} (j-1)f_j$. We denote by $M_{\mathbf{f},g}$ a uniform random bipartite map of genus $g$ with exactly $f_j$ faces of degree $2j$ for all $j \geq 1$ (the bound on $g$ guarantees the existence of such maps). We also write $|\mathbf{f}|=\sum_{j \geq 1} jf_j$ for the number of edges of such a map. On the other hand, a family of models $\left( \MM_{\q} \right)$ of random infinite bipartite planar maps was introduced in~\cite[Appendix C]{B18these}. These models are indexed by a set $\hQ$ of weight sequences $\q=(q_j)_{j \geq 1}$ (i.e. sequences of nonnegative numbers), and the distribution of $\MM_{\q}$ is characterized by the following Boltzmann property. There are constants $C_p$ for $p \geq 1$ such that for every finite bipartite map $m$ with a hole of perimeter $2p$ (see Section~\ref{sec:prelim} for precise definitions), we have
\[ \P \left( m \subset \MM_{\q} \right) = C_p \times \prod_{f \in m} q_{\deg(f)/2}, \]
where the product is over all internal faces of $m$. It was proved in~\cite[Appendix C]{B18these} that if such a map exists, its distribution is determined by $\q$, so the notation $\MM_{\q}$ makes sense. We will prove the following local convergence result.

\begin{thm}\label{thm_main}
	Let $(\mathbf{f}^{n})_{n \geq 1}$ be a sequence of face degree sequences, and let $(g_n)_{n \geq 1}$ be a sequence such that $0 \leq g_n \leq \frac{1}{2} \sum_{j \geq 1} (j-1)f_j^n$ for all $n \geq 1$. We assume that $|\mathbf{f}^n| \to +\infty$ when $n \to +\infty$ and that $\frac{f^n_j}{|\mathbf{f}^n|} \to \alpha_j$ for all $j \geq 1$, where $\sum_{j \geq 1} j \alpha_j=1$.
	We also assume $\frac{g_n}{|\mathbf{f}^n|} \to \theta$, where $0 \leq \theta < \frac{1}{2} \sum_{j \geq 1} (j-1) \alpha_j$.
	
	Then we have the convergence in distribution
	\[ M_{\ff^n, g_n} \xrightarrow[n \to +\infty]{(d)} \MM_{\q}\]
	for the local topology, where the weight sequence $\q$ depends only on $\theta$ and $\left( \alpha_j \right)_{j \geq 1}$. Moreover, this induces a bijection between, on the one hand, the parameters $\left( (\alpha_j)_{j \geq 1}, \theta \right)$ satisfying the assumptions above and, on the other hand, the weight sequences $\q$ such that $\MM_{\q}$ exists.
\end{thm}

This was also the main result of~\cite{BL22} under the additional assumption that $\sum_{j \geq 1} j^2 \alpha_j<+\infty$, which is equivalent to saying that the limit distribution of the degree of the root face of $M_{\ff^n, g_n}$ has finite expectation. On the other hand, the conditions required here are the minimal ones needed to obtain a limit with finite faces and finite vertex degrees: if $\sum_{j \geq 1} j \alpha_j < 1$, then the face incident to the root edge has a positive probability to become infinite in the limit. If $\theta=\frac{1}{2}\sum_{j \geq 1} (j-1) \alpha_j$, then by the Euler formula the expected inverse degree of the root vertex goes to $0$, which means that the degree of the root goes to infinity in probability.

\paragraph{Quenched local convergence.}
As a byproduct of the proof, we obtain a quenched local limit result, which is a slight reinforcement of Theorem~\ref{thm_main}. If $m_0$ is a finite bipartite planar map with a hole (see Section~\ref{sec:prelim} for the precise definition) and $m$ is a finite bipartite map, we denote by $\occ_{m_0}(m)$ the number of occurences of $m_0$ in $m$. In other words, this is the number of oriented edges $e$ of $m$ such that there is a neighbourhood of $e$ in $m$ that is isomorphic to $m_0$, in such a way that $e$ is matched with the root edge of $m_0$.

\begin{thm}\label{thm:quenched}
	Let $(\ff^n)_{n \geq 1}$ and $(g_n)_{n \geq 1}$ satisfy the assumptions of Theorem~\ref{thm_main}. Then for any finite planar bipartite map $m_0$ with a hole, we have the convergence in probability
	\[ \frac{\occ_{m_0} \left( M_{\ff^n,g_n} \right)}{2|\ff^n|} \xrightarrow[n \to +\infty]{} \P \left( m_0 \subset \MM_{\q} \right),  \]
	where $\q$ depends only on $\theta$ and $(\alpha_j)_{j \geq 1}$ and is the same as in Theorem~\ref{thm_main}.
\end{thm}

This was not proved in~\cite{BL22}, even under the light tail assumption. Note that Theorem~\ref{thm_main} directly implies the convergence of $\frac{1}{2|\ff^n|} \E \left[ \occ_{m_0} \left( M_{\ff^n,g_n} \right) \right]$, which is weaker. On the other hand, the analogue of Theorem~\ref{thm:quenched} was proved in~\cite{DS18} for a wide class of models of genus $0$ (that is, for critical Boltzmann maps without bipartiteness assumption).

\paragraph{Sketch of the proof.}
In~\cite[Theorem 2]{BL22}, it is proved (under the general assumptions of Theorem~\ref{thm_main}) that $\left( M_{\mathbf{f}^n,g_n} \right)$ is tight for the local topology and that any subsequential limit is of the form $\MM_{\QQ}$ for some random Boltzmann weight sequence $\QQ$. The next step of~\cite{BL22}, which is the only one requiring the assumption $\sum_j j^2\alpha_j<+\infty$, consists in proving that if one picks two roots $e_n, \oe_n$ independently in $M_{\mathbf{f}^n,g_n}$, the limiting (random) Boltzmann weights $\QQ$ and $\oQQ$ that we observe around $e_n$ and $\oe_n$ are almost surely the same. For face degrees with a light tail, this was shown in~\cite[Section 5]{BL22} by a surgery argument called the \emph{two-holes argument}. This argument relied on finding a common value attained by two random walks with a positive drift. However, if the face degrees have a heavy tail, the associated random walks may have large positive jumps, so it is possible that their ranges do not intersect.

Therefore, our proof that $\QQ=\oQQ$ will be completely different (but perhaps simpler) and consist in two steps: first, we observe that the neighbourhoods of $e_n$ and $\oe_n$ in $M_n$ must satisfy a certain joint spatial Markov property. We deduce from this property (Proposition~\ref{prop:markovian_pairs_are_independent}) that if these neighbourhoods converge jointly to two Boltzmann maps $M$ and $\oM$ with random weights $\QQ$ and $\oQQ$, then $M$ and $\oM$ are independent \emph{conditionally on $(\QQ,\oQQ)$}.
This will follow from appplying an intermediate result of~\cite{BL22} (Proposition~\ref{prop:Markovian_one_map} below) to the law of $M$ conditionally on $\oM$.
In the second step of the proof, we study the density of certain local patterns (see Figure~\ref{fig:m1_and_m2}) in $M$ and $\oM$: on the one hand, a simple combinatorial argument on $M_{\mathbf{f}^n,g_n}$ shows that this density must be the same around $e_n$ and $\oe_n$. On the other hand, the density of these patterns is sufficient to characterize $\QQ$ and $\oQQ$, which will show that $\QQ=\oQQ$.

\paragraph{Acknowledgement.} The author thanks the anonymous referee for useful remarks.

\section{Notation}\label{sec:prelim}

\paragraph{Maps.} A (finite or infinite) \emph{map} $M$ is a way to glue a finite or countable collection of finite oriented polygons, called the \emph{faces}, along their edges in a connected way, in such a way that any vertex is incident to finitely many faces. The maps that we consider will always be \emph{rooted}, i.e. equipped with a distinguished oriented edge called the \emph{root edge}. An infinite map $M$ is \emph{one-ended} if for any finite set $A$ of faces, the map $M \backslash A$ has only one infinite connected component.
%The face on the right of the root edge is the \emph{root face}, and the vertex at the start of the root edge is the \emph{root vertex}.
If the number of faces is finite, then $M$ is always homeomorphic to an orientable topological surface, so we can define the genus of $M$ as the genus of this surface. In particular, we call a map \emph{planar} if it has genus $0$.
%The \emph{dual map} of a map $m$ is the map $m^*$ whose vertices are the faces of $m$, and whose edges are the dual edges to those of $m$. We root $m^*$ at the oriented edge crossing the root edge of $m$ from left to right (see Figure~\ref{fig_map_dual}).
A \emph{bipartite} map is a rooted map where it is possible to color all vertices in black or white without any monochromatic edge. In a bipartite map, all faces have an even degree.
In what follows, we will only deal with bipartite maps even when it is not specified.

If $m$ is a (finite or infinite) map rooted at $e$ and if $r \geq 1$, we write $B_r(m,e)$ for the ball of radius $r$ around $e$ in $m$. More precisely, the map $B_r(m,e)$ consists of the faces of $m$ containing a vertex at graph distance at most $r$ from the starting point of $e$, along with all their vertices and edges. If $(m,e)$ and $(m',e')$ are two rooted maps, we define the \emph{local distance} between them by
\[ d_{\loc} \left( (m,e), (m',e') \right) = \left( 1+\min \{ r \geq 1 | B_r(m,e) \ne B_r(m',e') \}\right)^{-1}. \]
All the map convergences that we will consider will be for the topology induced by $d_{\loc}$. We also denote by $\B$ the space of infinite bipartite planar maps, equipped with $d_{\loc}$ and the associated Borel $\sigma$-algebra.

For every $\mathbf{f}=(f_j)_{j \geq 1}$ and $g\geq 0$, we will denote by $\B_g(\mathbf{f})$ the set of bipartite maps of genus $g$ with exactly $f_j$ faces of degree $2j$ for all $j\geq 1$. In particular, such a map exists if and only if $\sum_{j\geq 1} (j-1)f_j \geq 2g$, and in this case it has $|\mathbf{f}| =\sum_{j\geq 1} j f_j$ edges. We will denote by $\beta_g(\mathbf{f})$ the cardinality of $\B_g(\mathbf{f})$. Finally, for $j \geq 1$, we denote by $\mathbf{1}_j$ the sequence defined by $(\mathbf{1}_j)_i=1$ if $i=j$ and $0$ if not.

\paragraph{Maps with boundaries.} As often in the literature on local limits of maps, we will need to consider two different notions of bipartite maps with boundaries, that we call \emph{maps with a hole} and \emph{maps of multi-polygons}.

\begin{defn}
	A \emph{planar map with a hole} is a finite, bipartite planar map with a marked face (called its \emph{hole}) such that:
	\begin{itemize}
		\item
		the boundary of the hole is a simple cycle,
		\item
		the adjacency graph of the \emph{internal faces} (i.e. all the faces except the hole) is connected,
		\item
		the root edge may be any oriented edge of the map.
	\end{itemize}
	By convention, the trivial map consisting of two vertices joined by a single edge is a map with a hole (and no internal face). If $m$ is a map with a hole, we denote by $\partial m$  the boundary of the hole.
\end{defn}

\begin{defn}
	Let $\ell \geq 1$ and $p_1, p_2, \dots, p_{\ell} \geq 1$. A \emph{map of the $(2p_1, \dots, 2p_{\ell})$-gon} is a finite bipartite map with $\ell$ marked oriented edges $(e_i)_{1 \leq i \leq \ell}$, such that:
	\begin{itemize}
		\item
		$e_1$ is the root edge,
		\item
		the faces on the right of the $e_i$ are pairwise distinct,
		\item
		for all $1 \leq i \leq \ell$, the face on the right of $e_i$ has degree $2p_i$.
	\end{itemize}
	We will actually only need this definition for $\ell=1$ and $\ell=2$. The faces on the right of the marked edges are called \emph{external faces}, and the other ones are called \emph{internal faces}.
	We denote by $\B^{(p_1,p_2,\dots,p_{\ell})}_g(\mathbf{f})$ the set of bipartite maps of the $(2p_1,2p_2,\dots,2p_{\ell})$-gon of genus $g$ with internal face degrees given by $\textbf{f}$. We also denote by $\beta^{(p_1,p_2,\dots,p_{\ell})}_g(\mathbf{f})$ its cardinal.
\end{defn}
Note that in this second definition, we do not ask that the boundaries are simple or disjoint.

\paragraph{Map inclusion.}
Like in \cite{BL22}, our definition of map inclusion is the one tailored for the lazy peeling of~\cite{Bud15} (see also~\cite{C-StFlour}). Let $m$ be a map with a hole and let $M$ be a finite or infinite map. We write $m \subset M$ if $M$ can be obtained by gluing a (finite or infinite) map of the $|\partial m|$-gon in the hole of $m$. Note that the boundary of a map of the $2p$-gon is not necessarily simple, which may mean that two edges on the boundary of $m$ could correspond to two sides of the same edge of $M$. We also note that we always have $m \subset M$ if $m$ is the trivial one-edge map. %Finally, we call an infinite map $M$ planar if for any finite map $m$ with a hole such that $m \subset M$, the map $m$ is planar.

\paragraph{Infinite Boltzmann bipartite maps.}
Let $\q=(q_j)_{j \geq 1}$ be a sequence of nonnegative real numbers (we will use bold characters for sequences and usual characters for their terms). We say that a random infinite, one-ended bipartite planar map $M$ is \emph {$\q$-Boltzmann} if there are constants $\left( C_p \right)_{p \geq 1}$ such that, for every finite bipartite map $m$ with a hole of perimeter $2p$, we have
\[ \P \left( m \subset M \right)=C_p \prod_{f \in m} q_{\deg(f)/2},\]
where the product is over all internal faces of $m$. We also denote by $\hQ$ the set of weight sequences $\q$ for which such a random map exists. It is proved in~\cite[Appendix C]{B18these} that if $\q \in \hQ$, then there is a unique (in distribution) $\q$-Boltzmann infinite map. We denote it by $\MM_{\q}$. In particular, the numbers $C_p$ only depend on $\q$, so we can write them $C_p(\q)$. An explicit formula for $C_p(\q)$ (which will not be needed here) is provided in~\cite[Appendix C]{B18these}. In particular, this formula gives $C_1(\q)=1$ (the interpretation is that the trivial map consisting of two vertices joined by the root edge is always included in $\MM_{\q}$). Moreover, if $\QQ$ is a random variable with values in $\hQ$, we denote by $\MM_{\QQ}$ a random infinite map which has the law of $\MM_{\q}$ conditionally on $\QQ=\q$. That is, for any finite planar map $m$ with a hole, we have
\[ \P \left( m \subset \MM_{\QQ} \right) = \E \left[ C_p(\QQ) \prod_{f \in m} Q_{\deg(f)/2}\right]. \]
Finally, we recall that by~\cite[Proposition 7]{BL22}, the weight sequence $\q$ can almost surely be recovered from $\MM_{\q}$. More precisely, there is a measurable function $\tq : \B \to \hQ$ such that for all $\q \in \hQ$, we have $\tq \left( \MM_{\q} \right)=\q$ almost surely.
%In particular, if a random map $M$ is of the form $\MM_{\QQ}$, we have $\QQ=\tq(M)$ almost surely, and the distribution of $M$ is given by
%\[ \P \left( m \subset M \right) = \E \left[ C_p(\tq(M)) \prod_{f \in m} \tq_{\deg(f)/2}(M)\right]\]
%for any map $m$ with a hole.

\section{Setup}\label{sec:setup}

For the rest of the paper, we fix a sequence $\left( \mathbf{f}^n \right)_{n \geq 1}$ of face degree sequences and a sequence $(g_n)_{n \geq 1}$ satisfying the assumptions of Theorem~\ref{thm_main}, and we write $M_n$ for $M_{\mathbf{f}^n,g_n}$.

Theorem 2 of~\cite{BL22} states that the sequence $\left( M_n \right)$ is tight for the local topology and that any subsequential limit is of the form $\MM_{\mathbf{Q}}$ for some random sequence $\mathbf{Q}$ with $\mathbf{Q} \in \mathcal{Q}_h$ a.s.. In particular, let us focus on such a subsequence. Conditionally on $M_n$, let $e_n, \oe_n$ be two independent uniform oriented edges of $M_n$. By invariance of $M_n$ under uniform rerooting, both $\left( M_n, e_n \right)$ and $\left( M_n, \oe_n \right)$ have the same distribution as $M_n$. Hence, up to further extraction, we may assume the joint local convergence
\begin{equation}\label{eqn:double_local_convergence}
\left( \left( M_n, e_n \right), \left( M_n, \oe_n \right) \right) \xrightarrow[n \to +\infty]{}  \left( M, \oM \right)
\end{equation}
in distribution, where $M$ and $M'$ both have the same law as $\MM_{\QQ}$ for some random variable $\QQ \in \hQ$. In this setting, our main goal will be to prove the following result.

\begin{prop}\label{prop_main}
	Almost surely, we have $\tq(M)=\tq(M')$.
\end{prop}

\begin{proof}[Proof of Theorem~\ref{thm_main} using~Proposition~\ref{prop_main}]
	Proposition~\ref{prop_main} is exactly the same as~\cite[Proposition 30]{BL22}, but without the assumption $\sum_{j \geq 1} j^2 \alpha_j=+\infty$. In~\cite{BL22}, the end of the proof of the main Theorem using Proposition 30 (that is, the argument of Section 5.4) does not require the use of the tail assumption, so the exact same proof applies here.	

	It remains to check the claim that the application $\left( (\alpha_j)_{j \geq 1}, \theta \right) \to \q$ is bijective. For this, for $\q \in \hQ$ and $j \geq 1$, let
	\[ a_j(\q)=\frac{1}{j} \P \left( \mbox{the root face of $\MM_{\q}$ has degree $2j$} \right) \quad \mbox{and} \quad d(\q)=\E \left[ \frac{1}{\deg_{\MM_{\q}}(\rho)} \right],\]
	where $\rho$ is the root vertex of $\MM_{\q}$. As noted in~\cite[Corollary 31]{BL22} as a consequence of the Euler formula, under the assumptions of Theorem~\ref{thm_main}, we must have $a_j(\q)=\alpha_j$ for all $j \geq 1$ and $d(\q)=\frac{1}{2} \left( 1-2\theta-\sum_j \alpha_j \right)$.
	Therefore, the only parameters yielding $\MM_{\q}$ as a local limit are $\alpha_j=a_j(\q)$ and $\theta=\frac{1}{2} \left( 1-2d(\q)-\sum_{j \geq 1} a_j(\q) \right)$. This proves both injectivity and surjectivity.
\end{proof}

\section{Markovian pairs of infinite maps}

The goal of this section is to show that a certain joint Markov property for pairs of infinite random maps implies that the two elements of the pair are independent conditionally on their Boltzmann weights. This will constitute the first step of the proof of Proposition~\ref{prop_main}.

\begin{defn}
	\begin{enumerate}
		\item
		Let $M$ be a random infinite, one-ended bipartite planar map. We say that $M$ is \emph{Markovian} if for any bipartite map $m$ with a hole of perimeter $2p$  and internal face degrees given by $\v$, the probability $\P \left( m \subset M \right)$ only depends on $p$ and $\v$.
		\item
		Let $\left( M, \oM \right)$ be a pair of random infinite, one-ended bipartite planar maps. We say that $\left( M, \oM \right)$ is \emph{Markovian} if it satisfies the following property. For any two bipartite maps $m$ (resp. $\om$) with a hole of perimeter $2p$ (resp. $2\op$) and internal face degrees given by $\v$ (resp. $\ov$), the probability
		\[ \P \left( m \subset M \mbox{ and } \om \subset \oM \right)  \]
		only depends on $p$, $\op$, $\v$ and $\ov$.
	\end{enumerate}
\end{defn}

We first recall~\cite[Theorem 4]{BL22}, which provides a classification of Markovian maps.

\begin{prop}\label{prop:Markovian_one_map}
	For any Markovian map $M$, there is a random variable $\QQ$ with values in $\hQ$ such that $M$ has the law of $\MM_{\QQ}$.
\end{prop}

Here is a convenient way to write down this result: for any bounded measurable function $f$ from the space $\B$ of infinite bipartite planar maps to $\R$, we define $\tf : \hQ \to \R$ by $\tf(\q)=\E \left[ f(\MM_{\q}) \right]$. We also recall that there is a measurable function $\tq$ such that for any $\q \in \hQ$, we have $\tq(\MM_{\q})=\q$ a.s.. Proposition~\ref{prop:Markovian_one_map} says that if $M$ is Markovian, then there is $\QQ$ such that for any $f$, we have
\[ \E \left[ f(M) \right] = \E \left[ \tf(\QQ) \right]= \E \left[ \tf(\tq(M)) \right].  \]

We will now prove a natural extension of Proposition~\ref{prop:Markovian_one_map} to pairs of maps.

\begin{prop}\label{prop:markovian_pairs_are_independent}
	Let $(M,M')$ be a Markovian pair of random, infinite, one-ended bipartite planar maps. Then there is a pair of random variables $\left( \QQ, \oQQ \right)$ with values in $\mathcal{Q}_h$ such that conditionally on $\left( \QQ, \oQQ \right)$, the pair $(M,M')$ has the same joint distribution as two independent infinite Boltzmann maps with Boltzmann weights $\QQ$ and $\oQQ$.
\end{prop}

\begin{proof}
	We first provide a description of the law of two independent Boltzmann maps conditionally on their weights. Let $\left( \QQ, \oQQ \right)$ be a pair of random Boltzmann weight sequences and let $(M,M')$ be a pair of maps which, conditionally on $(\QQ, \oQQ)$, are independent infinite Boltzmann maps with weights $\QQ$ and $\oQQ$. Then we note that for any two bounded measurable functions $f,g : \B \to \R$, we have
	\begin{equation}\label{eqn:cond_independance_with_hats}
	\E \left[ f(M) g(\oM) \right] = \E \left[ \tf(\QQ) \tg(\oQQ) \right].
	\end{equation}
	Moreover, this identity characterizes the law of $(M,M')$.
	
	Therefore, let $\left( M, \oM \right)$ be a Markovian pair. The Proposition is equivalent to finding a pair $(\QQ,\oQQ)$ of variables on $\hQ$ such that~\eqref{eqn:cond_independance_with_hats} is satisfied for all $f$ and $g$. Roughly speaking, the proof will consist of applying Proposition~\ref{prop:Markovian_one_map} to the conditional distribution of $M$ given $\oM$.
	
	More precisely, let us fix $\om$ such that $\P \left( \om \subset \oM \right)>0$. For any $m$, the quantity
	\[ \P \left( m \subset M | \om \subset \oM \right) = \frac{\P \left( m \subset M \mbox{ and } \om \subset \oM \right)}{\P \left( \om \subset \oM \right)}  \]
	only depends on the perimeter and internal face degrees of $m$. This proves that the law of $M$ conditionally on $\om \subset \oM$ is Markovian. By the discussion following Proposition~\ref{prop:Markovian_one_map}, for any bounded measurable function $f : \mathcal{B} \to \R$, we have
	\begin{equation}\label{eqn:conditional_Markovian}
		\E \left[ f(M) | \om \subset \oM \right] = \E \left[ \tf \left( \tq(M) \right) | \om \subset \oM \right].
	\end{equation}
	Now consider a peeling algorithm on $\oM$ which almost surely explores the whole map (for example, always peel a boundary edge with minimal distance to the root). We write $\oM_n$ for the explored part of $\oM$ after $n$ peeling steps. By applying~\eqref{eqn:conditional_Markovian} to each possible value of $\oM_n$, we have
	\[ \E \left[ f(M) | \oM_n \right] = \E \left[ \tf \left( \tq(M) \right) | \oM_n \right]  \]
	for all $n \geq 1$. Since $\bigcup_{n \geq 1} \oM_n = \oM$, the $\sigma$-algebra generated by $\left( \oM_n \right)_{n \geq 1}$ is the $\sigma$-algebra generated by $\oM$, so by the martingale convergence theorem, letting $n \to +\infty$, we can write
	\begin{equation}\label{eqn:conditional_law_given_Mbar}
		\E \left[ f(M) | \oM \right] = \E \left[ \tf \left( \tq(M) \right) | \oM \right]
	\end{equation}
	almost surely, and a similar relation holds if we switch the roles of $M$ and $\oM$. Therefore, let $f,g$ be two bounded measurable functions from $\B$ to $\R$. Applying~\eqref{eqn:conditional_law_given_Mbar} first to $(M,\oM)$ and then to $(\oM,M)$, we obtain
	\[ \E \left[ f(M) g(\oM) \right] = \E \left[ \tf  \left( \tq(M) \right) g(\oM) \right] = \E \left[ \tf \left( \tq(M) \right) \tg \left( \tq(\oM) \right) \right]. \]
	This proves~\eqref{eqn:cond_independance_with_hats} with $\QQ=\tq(M)$ and $\oQQ=\tq(\oM)$, and therefore the proposition.
\end{proof}

\section{Equality of Boltzmann weights}

We now want to prove Proposition~\ref{prop_main}. The first step is an easy consequence of Proposition~\ref{prop:markovian_pairs_are_independent}.

\begin{lem}\label{lem:conditional_independence}
	In the setting of Section~\ref{sec:setup}, conditionally on the Boltzmann weights $\left( \tq(M), \tq(\oM) \right)$, the random maps $M$ and $\oM$ are independent.
\end{lem}

\begin{proof}
	By Proposition~\ref{prop:markovian_pairs_are_independent}, it is sufficient to prove that the pair $(M,\oM)$ is Markovian. For this, let $m$ (resp. $\om$) be a finite planar, bipartite map with a hole of perimeter $2p$ (resp. $2\op$) and internal face degrees given by $\v$ (resp. $\ov$). Along some subsequence, we can write
	\[
	\P \left( m \subset (M_n, e_n) \mbox{ and } \om \subset (M_n, \oe_n) \right) \xrightarrow[n \to +\infty]{} \P \left( m \subset M \mbox{ and } \om \subset \oM \right).
	\]
	On the other hand, we can decompose the left-hand side according to whether the copy of $m$ around $e_n$ and the copy of $\om$ around $\oe_n$ are face-disjoint or not. If they are not, it means that $\oe_n$ lies in the ball of radius $\diam(m)+\diam(\om)$ around $e_n$. However, by local convergence of $(M_n,e_n)$, the volume of this ball is tight as $n \to +\infty$. Since $\oe_n$ is chosen uniformly in $M_n$, the probability that this occurs goes to $0$. On the other hand, we have
	\[ \P \left( m \subset (M_n, e_n) \mbox{ and } \om \subset (M_n, \oe_n) \mbox{ face-disjointly} \right) = \frac{\beta_{g_n}^{(p,\op)}(\mathbf{f}^n-\v-\ov)}{2 |\mathbf{f}^n| \beta_{g_n}(\mathbf{f}^{n})}.
	\]
	Therefore, we must have
	\[ \frac{\beta_{g_n}^{(p,\op)}(\mathbf{f}^n-\v-\ov)}{2 |\mathbf{f}^n| \beta_{g_n}(\mathbf{f}^{n})} \xrightarrow[n \to +\infty]{} \P \left( m \subset M \mbox{ and } \om \subset \oM \right)\]
	along some subsequence. Since the left-hand side only depends on $(p,\op,\v, \ov)$, so does the right-hand side.
\end{proof}

We can now finish the proof of Proposition~\ref{prop_main} (and therefore of Theorem~\ref{thm_main}). The second step consists in computing in two different ways the probability to observe a certain pattern around the root in $M$ and $\oM$.

\begin{proof}[Proof of Proposition~\ref{prop_main}]
	We write $\QQ=\tq(M), \oQQ=\tq(\oM)$. We fix some $j \geq 1$, and denote by $m_j^1$ and $m_j^2$ the two maps with a hole of Figure~\ref{fig:m1_and_m2}. Both of these maps have a hole of perimeter $2$ and $m_j^1$ has exactly one internal face with degree $2j$, whereas $m_j^2$ has exactly two internal faces, each with degree $2j$.
	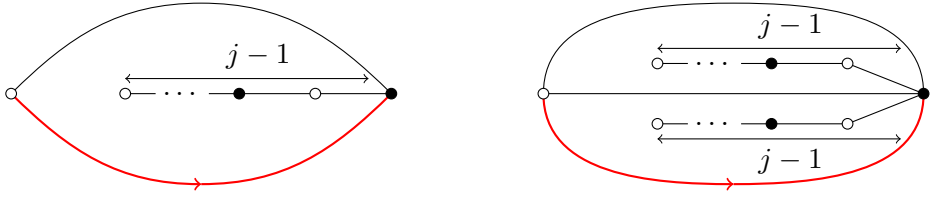
\begin{figure}
		\begin{center}
			\begin{tikzpicture}
				\draw[thick,red, ->](0,0)to[out=315,in=180](2.5,-1.2);
				\draw[thick,red](2.5,-1.2)to[out=0,in=225](5,0);
				\draw(0,0)to[out=45,in=180](2.5,1.2)to[out=0,in=135](5,0);
				\draw(5,0)--(4,0)--(3,0)--(2.6,0);
				\draw(1.5,0)--(1.9,0);
				\draw[<->](1.5,0.2)--(4.7,0.2);
				
				\draw(0,0)node[white]{};
				\draw(5,0)node[black]{};
				\draw(4,0)node[white]{};
				\draw(3,0)node[black]{};
				\draw(2.25,0)node[texte]{$\dots$};
				\draw(1.5,0)node[white]{};
				\draw(3.25,0.5)node[texte]{$j-1$};
				
				\begin{scope}[shift={(7,0)}]
					\draw[thick,red, ->](0,0)to[out=270,in=180](2.5,-1.2);
					\draw[thick,red](2.5,-1.2)to[out=0,in=270](5,0);
					\draw(0,0)to[out=90,in=180](2.5,1.2)to[out=0,in=90](5,0);
					\draw(0,0)--(5,0);
					\draw(5,0)--(4,0.4)--(3,0.4)--(2.6,0.4);
					\draw(1.5,0.4)--(1.9,0.4);
					\draw(5,0)--(4,-0.4)--(3,-0.4)--(2.6,-0.4);
					\draw(1.5,-0.4)--(1.9,-0.4);
					\draw[<->](1.5,0.6)--(4.7,0.6);
					\draw[<->](1.5,-0.6)--(4.7,-0.6);
					
					\draw(0,0)node[white]{};
					\draw(5,0)node[black]{};
					\draw(4,0.4)node[white]{};
					\draw(3,0.4)node[black]{};
					\draw(2.25,0.4)node[texte]{$\dots$};
					\draw(1.5,0.4)node[white]{};
					\draw(4,-0.4)node[white]{};
					\draw(3,-0.4)node[black]{};
					\draw(2.25,-0.4)node[texte]{$\dots$};
					\draw(1.5,-0.4)node[white]{};
					\draw(3.25,0.9)node[texte]{$j-1$};
					\draw(3.25,-0.9)node[texte]{$j-1$};
				\end{scope}
			\end{tikzpicture}
		\end{center}
		\caption{The maps $m_j^1$ (on the left) and $m_j^2$ (on the right). On each map, the root is in red and the hole is the face of degree $2$ on its right.}\label{fig:m1_and_m2}
	\end{figure}
	Now let $n$ be in a subsequence along which~\eqref{eqn:double_local_convergence} holds. Then we have
	\[ \P \left( m_j^2 \subset \left( M_n, e_n \right) \right) = \frac{\beta_g^{(1)}(\mathbf{f}^n-2 \cdot \mathbf{1}_j)}{\beta_g(\mathbf{f}^n)} = \frac{\beta_g(\mathbf{f}^n-2 \cdot \mathbf{1}_j)}{\beta_g(\mathbf{f}^n)} \]
	by closing the root edge. By local convergence, this implies
	\begin{equation}\label{eqn:convergence_m2_once}
		\frac{\beta_g(\mathbf{f}^n-2 \cdot \mathbf{1}_j)}{\beta_g(\mathbf{f}^n)} \xrightarrow[n \to +\infty]{} \P \left( m_j^2 \subset M \right) = \E \left[ C_1(\mathbf{Q}) Q_j^2 \right] = \E \left[ Q_j^2 \right]
	\end{equation}
	along our subsequence. By symmetry, the same is true if we replace $Q_j$ by $Q'_j$.
	
	On the other hand, if we have both $m^1_j \subset \left( M_n, e_n \right)$ and $m^1_j \subset \left( M_n, \oe_n \right)$ with face-disjoint copies of $m_j^1$, then the complement of the two copies of $m^1_j$ is a map with genus $g$, two boundaries of length $2$ and inner face degrees given by $\mathbf{f}^n-2 \cdot \mathbf{1}_j$. Therefore, we have
	\begin{align*}
		\P \left( m_j^1 \subset \left( M_n, e_n \right) \mbox{ and } m_j^1 \subset \left( M_n, \oe_n \right) \right) &= \frac{\beta_g^{(1,1)}(\mathbf{f}^n-2 \cdot \mathbf{1}_j)}{2 |\mathbf{f}^n| \beta_g(\mathbf{f}^n)}+O \left( \frac{1}{n}\right)\\
		&= \frac{\left( 2|\mathbf{f}^n|-2 \right)\beta_g(\mathbf{f}^n-2 \cdot \mathbf{1}_j)}{2 |\mathbf{f}^n| \beta_g(\mathbf{f}^n)}+O \left( \frac{1}{n}\right),
	\end{align*}
	where the $O \left( \frac{1}{n} \right)$ accounts for the probability that $e_n$ and $\oe_n$ are the same edge (possibly with different orientations). The second equality comes from closing the two boundaries, which leaves a marked edge in addition to the root. Hence, letting $n \to +\infty$ and using local convergence, we can write
	\[ \frac{\beta_g(\mathbf{f}^n-2 \cdot \mathbf{1}_j)}{\beta_g(\mathbf{f}^n)} \xrightarrow[n \to +\infty]{} \P \left( m_j^1 \subset M \mbox{ and } m_j^1 \subset \oM \right). \]
	By conditioning on $\left( \QQ, \oQQ \right)$ and using Lemma~\ref{lem:conditional_independence}, this becomes
	\begin{equation}\label{eqn:convergence_m1_twice}
		\frac{\beta_g(\mathbf{f}^n-2 \cdot \mathbf{1}_j)}{\beta_g(\mathbf{f}^n)} \xrightarrow[n \to +\infty]{} \E \left[ C_1(\QQ) Q_j C_1(\oQQ) Q'_j \right] = \E \left[ Q_j Q'_j \right].
	\end{equation}
	Combining~\eqref{eqn:convergence_m2_once} and~\eqref{eqn:convergence_m1_twice}, we finally find
	\[ \E \left[ \left( Q_j-Q'_j \right)^2 \right] = \E \left[ Q_j^2 \right] + \E \left[ (Q'_j)^2 \right] -2 \E \left[ Q_j Q'_j\right] =0, \]
	so $Q_j=Q'_j$ a.s.. Since this is true for all $j \geq 1$, this concludes the proof of Proposition~\ref{prop_main}.
\end{proof}

\begin{proof}[Proof of Theorem~\ref{thm:quenched}]
	The proof is a standard second moment argument combining Theorem~\ref{thm_main} and Proposition~\ref{prop:markovian_pairs_are_independent}. More precisely, we fix a finite planar map $m_0$ with a hole. Let $(\ff^n)_{n \geq 1}$ and $(g_n)_{n \geq 1}$ satisfy the assumptions of Theorem~\ref{thm_main}, and let $\q \in \hQ$ be the weight sequence provided by Theorem~\ref{thm_main}. In this proof, we \emph{do not} restrict the values of $n$ to a subsequence. Let $e_n, \oe_n$ be picked uniformly and independently among the $2|\ff^n|$ oriented edges of $M_n=M_{\ff^n,g_n}$. By Theorem~\ref{thm_main} and invariance of $M_n$ under uniform rerooting, we can write
	\begin{equation}\label{eqn:occurences_first_moment}
		\frac{\E \left[ \occ_{m_0} \left( M_n \right) \right]}{2|\ff^n|} = \P \left( m_0 \subset \left( M_n, e_n \right) \right) \xrightarrow[n \to +\infty]{} \P \left( m_0 \subset \MM_{\q} \right).
	\end{equation}
	On the other hand, we claim that the pair $\left( (M_n, e_n), (M_n, \oe_n) \right)$ converges jointly to a pair of independent copies of $\MM_{\q}$. Indeed, the pairs $\left( (M_n, e_n), (M_n, \oe_n) \right)$ are tight by tightness of both marginals and any subsequential limit $(M,M')$ is a Markovian pair by the same reasoning as the proof of Lemma~\ref{lem:conditional_independence}. Therefore, by Proposition~\ref{prop:markovian_pairs_are_independent}, the maps $M$ and $M'$ are independent conditionally on their Boltzmann weights. But by Theorem~\ref{thm_main}, these Boltzmann weights must be $\q$, so the only possible subsequential limit consists of two independent copies $\left( \MM_{\q}, \oMM_{\q} \right)$ of $\MM_{\q}$.
	
	Therefore, we can finally write
	\begin{eqnarray}
		\frac{\E \left[ \left( \occ_{m_0} \left( M_n \right) \right)^2 \right]}{\left( 2|\ff^n| \right)^2} &=& \P \left( m_0 \subset (M_n,e_n) \mbox{ and } m_0 \subset (M_n, \oe_n)\right) \nonumber \\
		&\xrightarrow[n \to +\infty]{}& \P \left( m_0 \subset \MM_{\q} \mbox{ and } m_0 \subset \oMM_{\q} \right) \nonumber \\
		&=& \P \left( m_0 \subset \MM_{\q} \right)^2.\label{eqn:occurences_second_moment}
	\end{eqnarray}
	Combining~\eqref{eqn:occurences_first_moment} and~\eqref{eqn:occurences_second_moment}, we find that the variance of $\frac{\occ_{m_0}(M_n)}{2|\ff^n|}$ goes to $0$ as $n \to +\infty$, which concludes the proof.
\end{proof}

\bibliographystyle{abbrv}
\bibliography{bibli}

\begin{thebibliography}{10}

\bibitem{AS03}
O.~Angel and O.~Schramm.
\newblock Uniform infinite planar triangulations.
\newblock {\em Comm. Math. Phys.}, 241(2-3):191--213, 2003.

\bibitem{Bud15}
T.~Budd.
\newblock The peeling process of infinite {B}oltzmann planar maps.
\newblock {\em Electronic Journal of Combinatorics}, 23, 06 2015.

\bibitem{BC16}
T.~Budd and N.~Curien.
\newblock Geometry of infinite planar maps with high degrees.
\newblock {\em Electron. J. Probab.}, 22:37 pp., 2017.

\bibitem{B18these}
T.~Budzinski.
\newblock {\em Cartes aléatoires hyperboliques.}
\newblock PhD thesis, Universit{\'e} Paris-Sud, 2018.

\bibitem{BL19}
T.~{Budzinski} and B.~{Louf}.
\newblock {Local limits of uniform triangulations in high genus}.
\newblock {\em Inventiones Mathematicae}, 223:1--47, 2021.

\bibitem{BL22}
T.~Budzinski and B.~Louf.
\newblock {Local limits of bipartite maps with prescribed face degrees in high
  genus}.
\newblock {\em The Annals of Probability}, 50(3):1059 -- 1126, 2022.

\bibitem{CurPSHIT}
N.~Curien.
\newblock Planar stochastic hyperbolic triangulations.
\newblock {\em Probability Theory and Related Fields}, 165(3):509--540, 2016.

\bibitem{C-StFlour}
N.~Curien.
\newblock Peeling random planar maps.
\newblock {\em Saint-Flour lecture notes}, 2019.

\bibitem{DS18}
M.~Drmota and B.~Stufler.
\newblock Pattern occurrences in random planar maps.
\newblock {\em Statistics and Probability Letters}, 158:108666, 2020.

\bibitem{St18}
R.~Stephenson.
\newblock Local convergence of large critical multi-type {G}alton--{W}atson
  trees and applications to random maps.
\newblock {\em Journal of Theoretical Probability}, 31(1):159--205, Mar 2018.

\end{thebibliography}

\end{document}